\documentclass[12pt, reqno]{amsart}
\usepackage{amssymb,amsthm,amsmath}
\usepackage[numbers,sort&compress]{natbib}
\usepackage{amssymb,amsmath}
\usepackage{amsfonts}
\usepackage{mathrsfs}
\usepackage{latexsym}
\usepackage{amssymb}
\usepackage{amsthm}
\usepackage{indentfirst}
\date{April 21, 2015}

\let\oldsection\section
\renewcommand\section{\setcounter{equation}{0}\oldsection}

\newtheorem{corollary}{Corollary}[section]
\newtheorem{theorem}{Theorem}[section]
\newtheorem{lemma}{Lemma}[section]
\newtheorem{proposition}{Proposition}[section]

\newtheorem{remark}{Remark}[section]

\allowdisplaybreaks
\begin{document}

\title[Global well-posedness of a tropical climate model]{Global well-posedness of strong solutions to a tropical climate model}


\author{Jinkai~Li}
\address[Jinkai~Li]{Department of Computer Science and Applied Mathematics, Weizmann Institute of Science, Rehovot 76100, Israel.}
\email{jklimath@gmail.com}

\author{Edriss~S.~Titi}
\address[Edriss~S.~Titi]{
Department of Mathematics, Texas A\&M University, 3368--TAMU, College Station, TX 77843-3368, USA. ALSO, Department of Computer Science and Applied Mathematics, Weizmann Institute of Science, Rehovot 76100, Israel.}
\email{titi@math.tamu.edu and edriss.titi@weizmann.ac.il}



\begin{abstract}
In this paper, we consider the Cauchy problem to the \textsc{tropical climate model} derived by Frierson--Majda--Pauluis in \cite{FRIMAJPAU}, which is a coupled system of the barotropic and the first baroclinic modes of the velocity and the typical midtropospheric temperature. The system considered in this paper has viscosities in the momentum equations, but no diffusivity in the temperature equation. We establish here the global well-posedness of strong solutions to this model. In proving the global existence of strong solutions, to overcome the difficulty caused by the absence of the diffusivity in the temperature equation, we introduce a new velocity $w$ (called the pseudo baroclinic velocity), which has more regularities than the original baroclinic mode of the velocity. An auxiliary function $\phi$, which looks like the effective viscous flux for the compressible Navier-Stokes equations, is also introduced to obtain the $L^\infty$ bound of the temperature. Regarding the uniqueness, we use the idea of performing suitable energy estimates at level one order lower than the natural basic energy estimates for the system.
\end{abstract}

\maketitle

\allowdisplaybreaks

\begin{center}
  Dedicated to Professor Peter Lax on the occasion of his 90th birthday
\end{center}

\section{Introduction}
\label{sec1}
The primary mode of the flow in the tropics is in the first baroclinic mode, that is the winds in the lower troposphere are of equal magnitude but with opposite sign to those in the upper troposphere. In many studies of the tropical atmospheric dynamics dating back to Gill \cite{GILL} and Matsuno \cite{MATSUNO}, the first baroclinic mode models were used. In these models, a typical midtropospheric temperature and the first baroclinic mode velocity are involved. However, as indicated in Majda--Biello \cite{MAJDABIELLO}, for the study of the tropical-extratropical interactions, where the transport of momentum between the barotropic and the baroclinic mode is an important effect, it is necessary to retain both the barotropic and baroclinic modes of the velocity. Taking the tropical-extratropical interactions factor into consideration, Frierson--Majda--Pauluis \cite{FRIMAJPAU} introduced a \textsc{Tropical Climate Model}, which keeps both the barotropic and the first baroclinic modes of the unknowns. The system in \cite{FRIMAJPAU} was derived from the inviscid primitive equations by performing a Galerkin truncation up to the first baroclinic mode.

In this paper, we consider the following \textsc{Tropical Climate Model} introduced by Frierson--Majda--Pauluis in \cite{FRIMAJPAU}:
\begin{eqnarray}
  &&\partial_tu+(u\cdot\nabla)u-\Delta u+\nabla p+\text{div}\,(v\otimes v)=0,\label{main1.1}\\
  &&\text{div}\, u=0,\label{main1.2}\\
  &&\partial_t v+(u\cdot\nabla) v-\Delta v+\nabla\theta+(v\cdot\nabla)u=0,\label{main1.3}\\
  &&\partial_t\theta+u\cdot\nabla\theta+\text{div}\, v=0,\label{main1.4}
\end{eqnarray}
in $\mathbb R^2$, where the unknowns are the vector fields $u=(u^1, u^2)$, $v=(v^1, v^2)$ and the scalar functions $\theta$ and $p$. Here, $u$ and $v$ are the barotropic mode and the first baroclinic mode of the velocity, respectively, while $\theta$ and $p$ denote the temperature and the pressure, respectively.

It should be pointed out that the original system derived in \cite{FRIMAJPAU} has no viscous terms in (\ref{main1.1}) and (\ref{main1.3}), in other words, the Laplacian terms are not involved in the system in \cite{FRIMAJPAU}; this is because it is derived from the inviscid primitive equations there. In this paper, we consider the viscous counterpart of the system in \cite{FRIMAJPAU}, i.e. system (\ref{main1.1})--(\ref{main1.4}), which can be derived by the same argument from the viscous primitive equations (but without any diffusivity in the temperature equation). Recalling that system (\ref{main1.1})--(\ref{main1.4}) is derived from the primitive equations, it is worth mentioning some mathematical results concerning the primitive equations, see, e.g., \cite{LTW92A,LTW92B,LTW95,CAOTITI1,CAOTITI2,KOB,KUZI1,KUZI2,CAOLITITI1,CAOLITITI2,CAOLITITI3,CAOLITITI4,CAOLITITI5,CINT,WONG} and the references therein. On the one hand, based on the results in \cite{LTW92A,LTW92B,LTW95,CAOTITI1,CAOTITI2,KOB,KUZI1,KUZI2,CAOLITITI1,CAOLITITI2,CAOLITITI3,CAOLITITI4,CAOLITITI5}, we know that, for the primitive equations with both full viscosities and full diffusivity, they have global weak solutions (but the uniqueness is still unknown), see \cite{LTW92A,LTW92B,LTW95}, and have a unique global strong solution, see \cite{CAOTITI1,KOB,KUZI1,KUZI2}. However, on the other hand, for the inviscid primitive equations, the results in \cite{CINT,WONG} show that they may develop finite-time singularity in both 2D and 3D. Comparing the global well-posedness results in \cite{CAOTITI1,KOB,KUZI1,KUZI2}, for the primitive equations with both full viscosities and full diffusivity, and the finite-time blowup results in \cite{CINT,WONG}, for the inviscid primitive equations, it is a natural question to ask if the partial viscosities and partial diffusivity can guarantee the global well-posedness for the primitive equations. As indicated in \cite{CAOLITITI3,CAOLITITI4,CAOLITITI5}, the merely horizontal viscosity can guarantee the global well-posedness of strong solutions to the primitive equations, as long as one still has the horizontal or vertical diffusivity, in other words, the horizontal viscosity is more crucial than the vertical one for the primitive equations. However, to our best knowledge, it is still unknown if the global well-posedness of strong solutions continues to hold for the primitive equations with full or only horizontal viscosities but without any diffusivity, and actually, this is also one of the motivations that thrust us to consider system (\ref{main1.1})--(\ref{main1.4}).

The aim of this paper is to prove the global well-posedness of strong solutions to system (\ref{main1.1})--(\ref{main1.4}). We consider the Cauchy problem, and thus compliment it with the following initial condition
\begin{equation}
  \label{ic}
  (u, v, \theta)|_{t=0}=(u_0, v_0, \theta_0).
\end{equation}

Throughout this paper, for $r\in[1,\infty]$ and positive integer $m$, we use $L^r(\mathbb R^2)$ and $H^m(\mathbb R^2)$ to denote the standard Lebessgue and Sobolev spaces on $\mathbb R^2$, respectively. For convenience, we always use $\|f\|_r$ to denote the $L^r(\mathbb R^2)$ norm of $f$.

The main result of this paper is the following:

\begin{theorem}
  \label{theorem}
  Suppose that $(u_0, v_0, \theta_0)\in H^1(\mathbb R^2)$, with $\text{div}\, u_0=0$ and $\theta_0\in L^\infty(\mathbb R^2)$. Then, there is a unique global strong solution $(u, v, \theta)$ to system (\ref{main1.1})--(\ref{main1.4}), subject to the initial condition (\ref{ic}), such that
  \begin{eqnarray*}
    &(u, v)\in C([0,T];H^1(\mathbb R^2))\cap L^2(0,T;H^2(\mathbb R^2)),\\
    &\theta\in C([0,T]; L^2(\mathbb R^2))\cap L^\infty(0,T;H^1(\mathbb R^2)),\\
    &(\partial_tu,\partial_tv,\partial_t\theta)\in L^2(0,T;L^2(\mathbb R^2)),\\
    &\theta\in L^\infty(0,T; L^\infty(\mathbb R^2)),\quad\nabla u\in L^1(0,T; L^\infty(\mathbb R^2)),
  \end{eqnarray*}
  for any positive time $T$.
\end{theorem}

\begin{remark}
  (i) Without the condition that $\theta_0\in L^\infty(\mathbb R^2)$, the same argument as in this paper still shows the global existence of strong solutions, which enjoy all the regularities in Theorem \ref{theorem}, except that $\theta\in L^\infty(0,T; L^\infty(\mathbb R^2))$.

  (ii) The condition that $\theta_0\in L^\infty(\mathbb R^2)$ is only used to guarantee the uniqueness of the strong solutions. We believe that such condition can be dropped if we work in the Lagrange coordinate in the proof of the uniqueness part of Theorem \ref{theorem}; however, to make the idea clear, we assume this condition throughout this paper.
\end{remark}

The key issue to prove the global existence part of Theorem \ref{theorem} is establishing the a priori $L^\infty(0,T; H^1(\mathbb R^2))$ estimate on $(u, v, \theta)$, for any positive time $T$. One can not use the standard energy approach to system (\ref{main1.1})--(\ref{main1.4}) to obtain the desired a priori estimate, because of the absence of the smoothing effect for the temperature equation. Actually, if doing so, in view of the presence of $\nabla\theta$ in (\ref{main1.3}), one has to establish the a priori $L^2(0,T; H^1(\mathbb R^2))$ estimate on $\theta$; however, since $\theta$ satisfies a transport type equation, we have to appeal to the a priori $L^1(0,T; Lip(\mathbb R^2))$ estimate on $u$, in other words, roughly speaking, one will end up with the following kind energy inequality:
$$
\frac{d}{dt}\|(u, v, \theta)\|_{H^1}^2+\|(u, v)\|_{H^2}^2\leq C\|\nabla u\|_\infty\|\nabla\theta\|_2^2+\mbox{other terms},
$$
which does not provide the desired estimate. To overcome this difficulty, we introduce a new unknown $w$ (we call it pseudo baroclinic velocity) as
$$
w:=v+\nabla(-\Delta)^{-1}\theta,
$$
which has more regularities than the original baroclinic velocity $v$. Actually, $w$ satisfies the following equation
\begin{equation*}
  \partial_tw+(u\cdot\nabla)w-\Delta w+(v\cdot\nabla)u+ \nabla(-\Delta)^{-1}\text{div}\,v+F =0,
\end{equation*}
with $F=[\mathcal R\otimes\mathcal R, u](\theta)$ being a commutator (see (\ref{2.4}), below) and $\mathcal R$ the Riesz transform. Thanks to the commutator estimates due to Coifman--Rochberg--Weiss \cite{COIROCWEI}, we have $F\in L^2(\mathbb R^2\times(0,T))$, and as a result, one can obtain the $L^\infty(0,T; H^1(\mathbb R^2))\cap L^2(0,T; H^2(\mathbb R^2))$ estimate on $(u, w)$ as well as the $L^\infty(0,T; L^4(\mathbb R^2))$ estimate on $\theta$, see Proposition \ref{secondenergy}, below. With the a priori estimates mentioned above in hand, combining the standard $H^1$ energy inequality for $\theta$ and a $t$-weighted $H^2$ energy inequality for $(u, w)$, and applying a logarithmic type Gronwall inequality (see Lemma \ref{loggronwall}, below), we can successfully achieve the $L^\infty(0,T; H^1(\mathbb R^2))$ estimate on $\theta$, see Proposition \ref{secondtheta}, below. With the help of the previous a priori estimates, one can successfully control the $L^1(0,T; L^\infty(\mathbb R^2))$ norm of $(\nabla u, \phi)$, where $\phi:=\text{div}\,v-\theta$ (which looks like the effective viscous flux for the compressible Navier-Stokes equations), based on which, one can further control the $L^\infty(\mathbb R^2\times(0,T))$ bound of $\theta$. With the aid of these a priori estimates, it is then standard to obtain the global strong solutions.

Regarding the uniqueness of the strong solutions, a usual way is to consider the difference equations between two solutions and then obtain some energy estimates for the resulting system of the difference at the level of the basic natural energy of the system. However, for system (\ref{main1.1})--(\ref{main1.4}), caused by the absence of the diffusivity term in the temperature equation, it does not seem to be suitable to perform the energy estimates at the
level of the basic natural energy of the system to prove the uniqueness. Motivated by the works Larios--Lunasin--Titi \cite{LarLunTiti} and Li--Titi--Xin \cite{LITITIXIN}, we perform the energy estimates at the level one order lower than the basic natural energy for the system, and successfully prove the uniqueness of the strong solutions established in Theorem \ref{theorem}.

The rest of this paper is arranged as follows: in the next section, section \ref{sec2}, we state some preliminary results which will be used in the following sections; in section \ref{secglobal}, we prove the global existence of strong solutions, while the uniqueness is proved in the last section, section \ref{secuniq}.

\section{Preliminaries}
\label{sec2}
The following lemma is about the commutator estimates, where the first inequality was proved in Coifman--Rochberg--Weiss \cite{COIROCWEI} (see Theorem I there), and the second one can be found in Feireisl--Novotn\'y \cite{FEINOV} (see (10.107) there), which is in the spirit of Coifman--Meyer \cite{COIMEY}.

\begin{lemma}[Commutator estimates]\label{commutator}
Let $\mathcal R=(\mathcal R_1,\mathcal R_2,\cdots,\mathcal R_N)$ be the Riesz transform on $\mathbb R^N$, and define the commutator
$$
[b,\mathcal R_i\mathcal R_j](f):=b\mathcal R_i\mathcal R_j(f)-R_iR_j(bf).
$$
Then, the following commutator estimates hold
$$
\|[b,\mathcal R_i\mathcal R_j](f)\|_{L^p(\mathbb R^N)}\leq C(N,p)[b]_{BMO(\mathbb R^N)}\|f\|_{ L^p(\mathbb R^N)},\quad p\in(1,\infty),
$$
and
$$
\|\nabla[b,\mathcal R_i\mathcal R_j](f)\|_{L^p(\mathbb R^N)}\leq C(N,p,q,r)\|\nabla b\|_{L^q(\mathbb R^N)}\|f\|_{L^r(\mathbb R^N)},
$$
for any $p,q,r\in(1,\infty)$, such that $\frac1p=\frac1q+\frac1r$. Here, the semi-norm $[\phi]_{BMO(\mathbb R^N)}$ is defined by
$$
[\phi]_{BMO(\mathbb R^N)}=\sup_B\frac{1}{|B|}\int_B|\phi-\phi_B|dx,
$$
where $\phi_B:=\frac{1}{|B|}\int_B\phi(x)dx$, and the supremum is taken over all balls in $\mathbb R^N$.
\end{lemma}

The next lemma is a logarithmic type Gronwall inequality, which will be employed to establish the a priori $H^1$ estimate of the temperature, see Proposition \ref{secondtheta}, below. Notably, similar type Gronwall inequality can be found in Cao--Li--Titi \cite{CAOLITITI3,CAOLITITI4} and Li--Titi \cite{LITITI}.

\begin{lemma}[Logarithmic Gronwall inequality]
  \label{loggronwall}
  Given a positive time $T$ and a positive constant $K$. Let $\alpha(t)$ and $\beta(t)$ be two nonnegative functions such that $\alpha,\beta\in L^1((0,T))$. Let $A(t)$ and $B(t)$ be integrable functions on $(0, T)$, with $A(t)\geq 1$ and $B(t)>0$, such that $A$ is absolutely continuous on
  $(0, T)$ and continuous on $[0, T)$. Suppose that, for any $t\in(0,T)$, we have
  \begin{equation}
  A'(t)+B(t)\leq K(\alpha(t)+\log B(t))A(t)+\beta(t).\label{B1}
  \end{equation}
  Then,
  $$
  \sup_{0\leq s\leq t}A(s)+\int_0^tB(s)ds\leq (2Q(t)+1)e^{Q(t)},
  $$
  for all $t\in[0,T)$, where
  $$
  Q(t)=\left(\log A(0)+K\|\alpha\|_{L^1((0,t))}+\|\beta\|_{L^1((0,t))}+ 2K^2t\right)e^{Kt}.
  $$
\end{lemma}

\begin{proof}
  Dividing both sides of inequality (\ref{B1}) by $A(t)$, and recalling that $A\geq1$, we have
  $$
  (\log A(t))'+\frac{B(t)}{A(t)}\leq K\left(\alpha(t)+\log\frac{B(t)}{A(t)}+\log A(t)\right)+\beta(t).
  $$
  One can easily verify that $\log \tau\leq 2\sqrt\tau,$ for any $\tau>0$, and thus it follows from the previous inequality and Young inequality that
  \begin{align*}
    (\log A(t))'+\frac{B(t)}{A(t)}\leq & K\left(\alpha(t)+2\sqrt{\frac{B(t)}{A(t)}}+\log A(t)\right)+\beta(t)\\
    \leq&K\log A(t)+K\alpha(t)+\beta(t)+\frac{B(t)}{2A(t)}+2K^2,
  \end{align*}
  which implies
  $$
   (\log A(t))'+\frac{B(t)}{2A(t)}\leq K\log A(t)+K\alpha(t)+\beta(t)+2K^2.
  $$
  By the Gronwall inequality, the above inequality implies
  \begin{align*}
    &\log A(t)+\int_0^t\frac{B(s)}{2A(s)}ds\\
    \leq& e^{Kt}\left(\log A(0)+K\|\alpha\|_{L^1((0,t))} +\|\beta\|_{L^1((0,t))}+ 2K^2t\right)=:Q(t),
  \end{align*}
  from which, we obtain
  \begin{equation*}
    A(t)\leq e^{Q(t)},\quad\int_0^t\frac{B(s)}{A(s)}ds\leq 2Q(t).
  \end{equation*}
  Thanks to these estimates, and noticing that $Q(t)$ is increasing in $t$, we deduce
  \begin{align*}
    &\sup_{0\leq s\leq t}A(s)+\int_0^t B(s)ds=\sup_{0\leq s\leq t}A(s)+\int_0^t \frac{B(s)}{A(s)}A(s)ds\\
    \leq&\sup_{0\leq s\leq t}e^{Q(s)}+\int_0^t\frac{B(s)}{A(s)}e^{Q(s)}ds\leq e^{Q(t)}\left(1+\int_0^t\frac{B(s)}{A(s)}ds\right)\leq e^{Q(t)}(1+2Q(t)),
  \end{align*}
  proving the conclusion.
\end{proof}

\section{Global existence}
\label{secglobal}

In this section, we establish the global existence of strong solutions to the Cauchy problem of system (\ref{main1.1})--(\ref{main1.4}). To this end, we introduce a regularized system by adding a diffusivity term to the temperature equation. The regularized system can be globally solved by standard approach. To obtain the global strong solution to the original system, the key issue is to establish some sufficiently high order a priori estimates on the solutions to the regularized system that are independent of the regularization parameter.

The regularized system we use in this section is as follows
\begin{eqnarray}
  &&\partial_tu+(u\cdot\nabla)u-\Delta u+\nabla p+\text{div}\,(v\otimes v)=0,\label{1.1}\\
  &&\text{div}\, u=0,\label{1.2}\\
  &&\partial_t v+(u\cdot\nabla) v-\Delta v+\nabla\theta+(v\cdot\nabla)u=0,\label{1.3}\\
  &&\partial_t\theta+u\cdot\nabla\theta-\varepsilon\Delta\theta+\text{div}\, v=0,\label{1.4}
\end{eqnarray}
where $\varepsilon$ is a positive constant.

For the regularized system (\ref{1.1})--(\ref{1.4}) we have the following global well-posedness result:

\begin{proposition}\label{glob}
  Suppose that $(u_0, v_0, \theta_0)\in H^2(\mathbb R^2)$, with $\text{div}\, u_0=0$. Then, there is a unique global strong solution $(u, v,\theta)$ to system (\ref{1.1})--(\ref{1.4}), with initial data $(u_0, v_0, \theta_0)$, such that
  \begin{eqnarray*}
    &&(u, v, \theta)\in C([0,\infty); H^2(\mathbb R^2))\cap L^2_{\text{loc}}([0,\infty); H^3(\mathbb R^2)),\\
    &&(\partial_tu,\partial_tv,\partial_t\theta)\in L^2_{\text{loc}}([0,\infty);H^1(\mathbb R^2)).
  \end{eqnarray*}
\end{proposition}

Since the basic energy estimate for the strong solutions to system (\ref{1.1})--(\ref{1.4}) will be used in the proof of the global existence part of Proposition \ref{glob}, let us first state the basic energy estimate in the following:

\begin{proposition}[Basic energy estimate] Let $(u, v, \theta)$ be a strong solution to system (\ref{1.1})--(\ref{1.4}) on the interval $[0,\mathcal{T})$ with initial data $(u_0, v_0, \theta_0)\in H^2(\mathbb R^2)$. Then, we have the following, uniform in time, energy estimate
  \label{basicenergy}
  \begin{equation*}
    \|(u,v,\theta)\|_2^2(t)+2\int_0^t\|(\nabla u,\nabla v,\sqrt\varepsilon\nabla\theta)\|_2^2(s)ds=\|(u_0, v_0, \theta_0)\|_2^2,
  \end{equation*}
  for all $t\in [0,\mathcal{T})$.
\end{proposition}
\begin{proof}
  Thanks to the regularity of strong solutions, we can take the $L^2(\mathbb R^2)$ inner product to equations (\ref{1.1}), (\ref{1.3}) and (\ref{1.4}) with $u, v$ and $\theta$, respectively, summing the resultants up and integration by parts, we have
  $$
  \frac{1}{2}\frac{d}{dt}\|(u, v,\theta)\|_2^2+\|(\nabla u, \nabla v,\sqrt\varepsilon\theta)\|_2^2=0, \quad {\mbox{for all}} \,\, t\in (0,\mathcal{T}),
  $$
  from which the conclusion of the proposition follows by integrating with respect to $t$.
\end{proof}

Now, let's give the proof of Proposition \ref{glob}.

\begin{proof}[\textbf{Proof of Proposition \ref{glob}}]
  Local existence and uniqueness of strong solutions to system (\ref{1.1})--(\ref{1.4}), with initial data $(u_0, v_0, \theta_0)$, can be proven in the standard way: given a positive time $T$, for any $(u, v, \theta)\in L^2(0,T; H^3)\cap C([0,T]; H^2)$, there is a unique solution $(U, V, \Theta)\in L^2(0,T; H^3)\cap C([0,T]; H^2)$, with $(\partial_tu,\partial_tv,\partial_t\theta)\in L^2(0,T;H^1(\mathbb R^2))$, to the following linear system
  \begin{eqnarray*}
    &&\partial_tU-\Delta U+\nabla P=-(u\cdot\nabla)u-\text{div}\,(v\otimes v), \\
    &&\text{div}\, U=0,\\
    &&\partial_tV-\Delta V=-(u\cdot\nabla)v-\nabla\theta-(v\cdot\nabla)u,\\
    &&\partial_t\Theta-\varepsilon\Delta\Theta=- u\cdot\nabla\theta-\text{div}\,v,
  \end{eqnarray*}
  defined on $\mathbb R^2\times(0,T)$, subject to the initial data $(u_0, v_0, \theta_0)$. By standard energy estimates, one can show that the solution map to the above linear system, $\Re: (u, v, \theta)\rightarrow(U, V, \Theta)$, is a contraction map from $L^2(0,T; H^3)\cap C([0,T]; H^2)$ into itself, for sufficiently small positive time $T$, depending only on $(u_0, v_0, \theta_0)$,
  and thus, by the contraction map principle, there is a unique fixed point $(u, v, \theta)$ to the solution map $\Re$, which is the unique local strong solution to system (\ref{1.1})--(\ref{1.4}), with initial data $(u_0, v_0, \theta_0)$. The proof is lengthy but standard, and thus it is omitted here.

  Now, we prove the global existence of strong solutions. Let $T^*>0$ be the maximal existence time of the unique strong solution $(u, v, \theta)$ to system (\ref{1.1})--(\ref{1.4}), with initial data $(u_0, v_0, \theta_0)$. We need to show that $T^*=\infty$. Suppose, by contradiction, that $T^*<\infty$.
  Applying the operator $\nabla$ to (\ref{1.1}), (\ref{1.3}) and (\ref{1.4}), and taking the $L^2(\mathbb R^2)$ inner product to the resultants with $-\nabla\Delta u, -\nabla\Delta v$ and $-\nabla\Delta\theta$, respectively, then it follows from integration by parts and the Cauchy-Schwarz inequality that
  \begin{align*}
    &\frac12\frac{d}{dt}\|(\Delta u, \Delta v,\Delta\theta)\|_2^2+ \|(\nabla\Delta u,\nabla\Delta v,\sqrt\varepsilon\nabla\Delta\theta)\|_2^2\\
    \leq&C\int_{\mathbb R^2}[(|u|+|v|)(|\nabla^2u|+|\nabla^2v|+|\nabla^2 \theta|)+|\nabla u|^2+|\nabla v|^2+|\nabla\theta|^2\\
    &+|\nabla^2\theta|+|\nabla^2v|](|\nabla\Delta u|+|\nabla\Delta v|+|\nabla\Delta\theta|)dx=:I,
  \end{align*}
  for all $t\in (0,T^*)$.
  By the H\"older, Ladyzhenskaya and Young inequalities, the quantity $I$ can be estimated as follows
  \begin{align*}
    I\leq&C(\|(u, v)\|_4\|(\nabla^2u,\nabla^2v,\nabla^2\theta)\|_4+\|(\nabla u,\nabla v,\nabla\theta)\|_4^2\\
    &+\|(\Delta v,\Delta\theta)\|_2)\|(\nabla\Delta u, \nabla\Delta v,\nabla\Delta\theta)\|_2\\
    \leq&C(\|(u,v)\|_2^{\frac12}\|(\nabla u,\nabla v)\|_2^{\frac12}\|(\Delta u,\Delta v,\Delta\theta)\|_2^{\frac12}\|(\nabla\Delta u,\nabla\Delta v,\nabla\Delta\theta)\|_2^{\frac12}\\
    &+\|(\Delta v,\Delta\theta)\|_2+\|(\nabla u,\nabla v,\nabla\theta)\|_2\|(\Delta u,\Delta v,\Delta\theta)\|_2)\|(\nabla\Delta u, \nabla\Delta v,\nabla\Delta\theta)\|_2\\
    \leq&\frac12\|(\nabla\Delta u, \nabla\Delta v,\sqrt\varepsilon\nabla\Delta\theta)\|_2+C_\varepsilon(1+\|(u, v)\|_2^2\|(\nabla u,\nabla v)\|_2^2\\
    &+\|(\nabla u,\nabla v,\nabla\theta)\|_2^2)\|(\Delta u, \Delta v,\Delta\theta)\|_2^2,
  \end{align*}
  for all $t\in (0,T^*)$.
  Therefore, we have
  \begin{align*}
    &\frac{d}{dt}\|(\Delta u, \Delta v,\Delta\theta)\|_2^2+ \|(\nabla\Delta u,\nabla\Delta v,\sqrt\varepsilon\nabla\Delta\theta)\|_2^2\\
    \leq&C_\varepsilon(1+\|(u, v)\|_2^2\|(\nabla u,\nabla v)\|_2^2 +\|(\nabla u,\nabla v,\nabla\theta)\|_2^2)\|(\Delta u, \Delta v,\Delta\theta)\|_2^2,
  \end{align*}
  from which, by the Gronwall inequality, and using the basic energy identity for the solutions to system (\ref{1.1})--(\ref{1.4}), i.e. Proposition \ref{basicenergy}, with $\mathcal{T}=T^*$, one arrives at
  \begin{align*}
    &\sup_{0\leq s\leq t}\|(\Delta u, \Delta v,\Delta\theta)\|_2^2
    +\int_0^t\|(\nabla\Delta u,\nabla\Delta v,\sqrt\varepsilon\nabla\Delta\theta)\|_2^2ds\\
    \leq&e^{C_\varepsilon\int_0^t(1+\|(u, v)\|_2^2\|(\nabla u,\nabla v)\|_2^2+\|(\nabla u,\nabla v,\nabla\theta)\|_2^2)ds}\|(\Delta u_0,\Delta v_0,\Delta\theta_0)\|_2^2\\
    \leq&e^{C_\varepsilon(t+\|(u_0, v_0, \theta_0)\|_2^2+\|(u_0, v_0, \theta_0)\|_2^4)}\|(\Delta u_0,\Delta v_0,\Delta\theta_0)\|_2^2,
  \end{align*}
  for every $t\in(0,T^*)$. Thanks to the above estimate and using Proposition \ref{basicenergy},  with $\mathcal{T}=T^*$, it follows from the elliptic estimates that
  \begin{align*}
    &\sup_{0\leq t<T^*}\|(u, v, \theta)\|_{H^2}^2+\int_0^{T^*}\|(\nabla u, \nabla v, \nabla\theta)\|_{H^2}^2ds\leq C_\varepsilon<\infty,
  \end{align*}
  for some positive constant $C_\varepsilon$ depending only on $\varepsilon$ and the initial norm $\|(u_0, v_0, \theta_0)\|_{H^2}$.
  As a result, one can extend the strong solution $(u, v, \theta)$ beyond the time $T^*$, which contradicts to the definition of $T^*$. This contradiction implies that $T^*=\infty$, and consequently, we obtain a global strong solution. This completes the proof.
\end{proof}

We will establish several uniform in $\varepsilon$ estimates on the unique strong solution $(u, v, \theta)$ to system (\ref{1.1})--(\ref{1.4}), with initial data $(u_0, v_0, \theta_0)\in H^2(\mathbb R^2)$.
Since the $L^2$ type estimate has been included in Proposition \ref{basicenergy} (which by the above proof is now valid for $\mathcal{T}= \infty$), let us continue with the $H^1$ type estimates.
As explained in the introduction, due to the absence of the diffusivity term in the temperature equation, one can not establish the $H^1$ estimates on $(u, v, \theta)$ in the standard way, in other words, one can not obtain the desired estimates by just tying to multiply the equations by $(-\Delta u,-\Delta v,-\Delta \theta)$. To obtain the desired $H^1$ estimate, we introduce some auxiliary functions.
Let's define the function
$$
\Phi=(\Phi_1,\Phi_2)=\nabla(-\Delta)^{-1}\theta,
$$
in other words, $\Phi_i$ is the unique solution to
\begin{equation}
  -\Delta\Phi_i=\partial_i\theta,\quad i=1,2.\label{2.1}
\end{equation}

Applying the operator $\nabla(-\Delta)^{-1}$ to equation (\ref{1.4}) yields
\begin{equation}\label{2.2}
  \partial_t\Phi+(u\cdot\nabla)\Phi-\varepsilon\Delta\Phi +\nabla(-\Delta)^{-1}\text{div}\,v+F=0,
\end{equation}
where $F=(F_1, F_2)$ is given by
\begin{equation}
  \label{2.3}
  F=\nabla(-\Delta)^{-1}(u\cdot\nabla\theta)-(u\cdot\nabla)\Phi.
\end{equation}
Recalling that $\text{div}\, u=0$, one may calculate
\begin{align}
  F_i=&\partial_i(-\Delta)^{-1}\text{div}\,(u\theta)-u\cdot\nabla\Phi_i= \partial_i(-\Delta)^{-1}\partial_j(u^j\theta)-u^j\partial_j\partial_i(-\Delta)^{-1} \theta\nonumber\\
  =&\mathcal R_i\mathcal R_j(u^j\theta)-u^j\mathcal R_i\mathcal R_j(\theta)=[ \mathcal R_i\mathcal R_j, u^j](\theta),\quad i=1,2, \label{2.4}
\end{align}
where $\mathcal R=(\mathcal R_1,\mathcal R_2)=\nabla(-\Delta)^{\frac12}$ is the Riesz transform, and $[\mathcal R_i\mathcal R_j, u^j]$ is the commutator operator defined by
$$
[\mathcal R_i\mathcal R_j, u^j](f)=\mathcal R_i\mathcal R_j(u^jf)-u^j\mathcal R_i\mathcal R_j(f).
$$

Recalling the definition of $\Phi$, i.e. (\ref{2.1}), one can rewrite equation (\ref{1.3}) as
$$
\partial_t v+(u\cdot\nabla)v-\Delta(v+\Phi)+(v\cdot\nabla) u=0.
$$
Multiplying equation (\ref{2.2}) by $\frac{1}{1-\varepsilon}$, assuming $\varepsilon\in(0,1)$, summing the resultant with the above equation, and introducing a new velocity $w$, which is called pseudo baroclinic velocity, as
\begin{equation}
  \label{2.5}
  w=v+\frac{\Phi}{1-\varepsilon},
\end{equation}
then $w$ satisfies
\begin{equation}
  \label{2.6}
  \partial_tw+(u\cdot\nabla)w-\Delta w+(v\cdot\nabla)u+\frac{1}{1-\varepsilon} (\nabla(-\Delta)^{-1}\text{div}\,v+F)=0,
\end{equation}
where $F$ is the function defined by (\ref{2.3})--(\ref{2.4}).

It turns out that the $H^1$ estimate for $\theta$ depends on that for $(u, w)$. Thus, we work on the $H^1$ estimate for $(u, w)$ first, and one has the following proposition.

\begin{proposition}[$H^1$ of the velocity]\label{secondenergy}
Let $(u, v, \theta)$ be the unique global strong solution to system (\ref{1.1})--(\ref{1.4}), with initial data $(u_0, v_0, \theta_0)\in H^2(\mathbb R^2)$. Let $\varepsilon\in(0,\frac12)$ and $w$ be the function defined by (\ref{2.5}). Then, we have the following estimate, for all $t \in [0,\infty)$,
  \begin{align*}
    &\sup_{0\leq s\leq t}(\|(\nabla u,\nabla w)\|_2^2+\|\theta\|_4^2)(s)+ \int_0^t\|(\Delta u,\Delta w,\partial_tu,\partial_t w)\|_2^2(s)ds\leq S_1(t),
  \end{align*}
  where $S_1(t)$ is an explicit nondecreasing continuous function on $[0,\infty)$, which depends only on the initial norm $\|(u_0, v_0)\|_{H^1}+\|\theta_0\|_2+\|\theta_0\|_4$, in a continuous manner, and is independent of $\varepsilon$.
\end{proposition}

\begin{proof}
  Define the functions $f_1$ and $f_2$ as
  \begin{eqnarray*}
    &&f_1=(u\cdot\nabla)w+(v\cdot\nabla) u+\frac{1}{1-\varepsilon}(\nabla(-\Delta)^{-1}\text{div}\, v+F), \\
    &&f_2=(u\cdot\nabla) u+\text{div}(v\otimes v),
  \end{eqnarray*}
  where $F$ is the function given by (\ref{2.3})--(\ref{2.4}). With the notations $f_i, i=1,2$, one can rewrite equations (\ref{1.1}) and (\ref{2.6}) as
  \begin{eqnarray}
    &&\partial_t u-\Delta u+\nabla p+f_2=0,\label{2.10}\\
    &&\partial_tw-\Delta w+f_1=0. \label{2.11}
  \end{eqnarray}

  By the
  elliptic estimates and the commutator estimates (Lemma \ref{commutator}), one has
  $$
  \|\nabla(-\Delta)^{-1}\text{div}\, v\|_2\leq C\|v\|_2,\quad\|F\|_2\leq C[u]_{BMO(\mathbb R^2)}\|\theta\|_2\leq C\|\nabla u\|_2\|\theta\|_2,
  $$
  here, we have used the fact that $[u]_{BMO(\mathbb R^2)}\leq C\|\nabla u\|_2$ (see, e.g., section 5.8 of \cite{EVANS}).
  Thanks to these estimates, by the H\"older, Ladyzhenskaya and Cauchy-Schwarz inequalities, we have
  \begin{align}
    \|f_1\|_2^2\leq&\int_{\mathbb R^2}(|u|^2|\nabla w|^2+|v|^2|\nabla u|^2+4|\nabla(-\Delta)^{-1}\text{div}\,v|^2+4|F|^2)dx\nonumber\\
    \leq&\|u\|_4^2\|\nabla w\|_4^2+\|v\|_4^2\|\nabla u\|_4^2+C(\|v\|_2^2+\|\nabla u\|_2^2\|\theta\|_2^2)\nonumber\\
    \leq&C(\|u\|_4^2\|\nabla w\|_2\|\Delta w\|_2+\|v\|_4^2\|\nabla u\|_2 \|\Delta u\|_2+\|v\|_2^2+\|\nabla u\|_2^2\|\theta\|_2^2)\nonumber\\
    \leq&\frac18\|(\Delta u,\Delta w)\|_2^2+C\|(u,v)\|_4^4\|(\nabla u,\nabla w)\|_2^2\nonumber\\
    &+C(1+\|\nabla u\|_2^2)(\|v\|_2^2+\|\theta\|_2^2).\label{2.7}
  \end{align}
  Recalling the definitions of $\Phi$ and $w$, i.e. (\ref{2.1}) and (\ref{2.5}), by the elliptic estimates and the Ladyzhenskaya inequality, one has
  \begin{equation}
    \label{2.8}
    \|\nabla v\|_4\leq\|\nabla w\|_4+2\|\nabla\Phi\|_4\leq C(\|\nabla w\|_2^{\frac12}\|\Delta w\|_2^{\frac12}+\|\theta\|_4).
  \end{equation}
  Using (\ref{2.8}), it follows from the H\"older, Ladyzhenskaya and Cauchy inequalities that
  \begin{align}
    \|f_2\|_2^2\leq&\int_{\mathbb R^2}(|u|^2|\nabla u|^2+|\text{div}(v\otimes v)|^2)dx\nonumber\\
    \leq&(\|u\|_4^2\|\nabla u\|_4^2+4\|v\|_4^2\|\nabla v\|_4^2)\leq 4\|(u, v)\|_4^2\|(\nabla u,\nabla v)\|_4^2\nonumber\\
    \leq&C\|(u, v)\|_4^2(\|\nabla u\|_2\|\Delta u\|_2+\|\nabla w\|_2\|\Delta w\|_2+\|\theta\|_4^2)\nonumber\\
    \leq&\frac18\|(\Delta u,\Delta w)\|_2^2+C(1+\|(u, v)\|_4^4)(\|(\nabla u,\nabla w)\|_2^2+\|\theta\|_4^2). \label{2.9}
  \end{align}

  Taking the $L^2$ inner products to equations (\ref{2.10}) and (\ref{2.11}) with $\partial_tu-\Delta u$ and $\partial_tw-\Delta w$, respectively, and summing the resultants up, then it follows from integration by parts and the Cauchy-Schwarz inequality that
  \begin{align}
    &\frac{d}{dt}\|(\nabla u,\nabla w)\|_2^2+\|(\Delta u, \Delta w,\partial_t u,\partial_tw)\|_2^2\nonumber\\
    =&-\int_{\mathbb R^2}[f_1\cdot(\partial_tw-\Delta w)+f_2(\partial_tu-\Delta u)]dx\nonumber\\
    \leq&\frac14(\|\Delta u\|_2^2+\|\partial_tu\|_2^2+\|\Delta w\|_2^2+ \|\partial_tw\|_2^2)+2(\|f_1\|_2^2+\|f_2\|_2^2)\nonumber\\
    \leq&\frac{1}{2}\|(\Delta u, \Delta w,\partial_t u,\partial_tw)\|_2^2+C(1+\|\nabla u\|_2^2)(\|v\|_2^2+\|\theta\|_2^2)\nonumber\\
    &+C(1+\|(u, v)\|_4^4)(\|(\nabla u,\nabla w)\|_2^2+\|\theta\|_4^2). \label{2.12}
  \end{align}

  Multiplying equation (\ref{1.4}) by $|\theta|^2\theta$, and integrating over $\mathbb R^2$, then it follows from integration by parts, (\ref{2.8}) and the Young inequality that
  \begin{align*}
    &\frac14\frac{d}{dt}\|\theta\|_4^4+3\varepsilon\int_{\mathbb R^2}|\theta|^2|\nabla\theta|^2 dx\\
    =&-\int_{\mathbb R^2}\text{div}\,v|\theta|^2\theta dx\leq\|\nabla v\|_4\|\theta\|_4^3\\
    \leq&C(\|\nabla w\|_2^{\frac12}\|\Delta w\|_2^{\frac12}+\|\theta\|_4)\|\theta\|_4^3,
  \end{align*}
  which implies
  \begin{align*}
    \frac{d}{dt}\|\theta\|_4^2\leq C(\|\nabla w\|_2^{\frac12}\|\Delta w\|_2^{\frac12}+\|\theta\|_4)\|\theta\|_4
    \leq \frac14\|\Delta w\|_2^2+C(\|\nabla w\|_2^2+\|\theta\|_4^2),
  \end{align*}
  where, in the last step, the Young inequality has been used.

  Adding the above inequality with (\ref{2.12}) up, and using the Ladyzhenskaya inequality, one obtains
  \begin{align*}
    &\frac{d}{dt}(\|(\nabla u,\nabla w)\|_2^2+\|\theta\|_4^2)+\frac14\|(\Delta u,\Delta w,\partial_tu,\partial_t w)\|_2^2\\
    \leq&C(1+\|(u,v)\|_4^4)(\|(\nabla u,\nabla w)\|_2^2+\|\theta\|_4^2) +C(1+\|\nabla u\|_2^2)\|(v,\theta)\|_2^2\\
    \leq&C(1+\|(u,v)\|_2^2\|(\nabla u,\nabla v)\|_2^2)(\|(\nabla u,\nabla w)\|_2^2+\|\theta\|_4^2) \\
    &+C(1+\|\nabla u\|_2^2)\|(v,\theta)\|_2^2,
  \end{align*}
  from which, integrating in $t$ and using Proposition \ref{basicenergy}, the conclusion follows.
\end{proof}

Now, we can work on the $H^1$ estimate on $\theta$. As shown in the following proposition, a $t$-weight higher order estimates on $(v, w)$ are also involved.

\begin{proposition}[$H^1$ of the temperature]\label{secondtheta} Let $(u, v, \theta)$ be the unique global strong solution to system (\ref{1.1})--(\ref{1.4}), with initial data $(u_0, v_0, \theta_0)\in H^2(\mathbb R^2)$. Let $\varepsilon\in(0,\frac12)$ and $w$ be the function defined by (\ref{2.5}). Then, we have the following estimate, for all $t \in [0,\infty)$,
  \begin{align*}
    \sup_{0\leq s\leq t}(\|\nabla\theta\|_2^2(s)+s\|(\Delta u,\Delta w)\|_2^2(s))+\int_0^t(\varepsilon\|\Delta\theta\|_2^2+ s\|(\|\nabla\Delta u, \nabla\Delta w)\|_2^2)ds\leq S_2(t),
  \end{align*}
  where $S_2(t)$ is an explicit nondecreasing continuous function on $[0,\infty)$, depending only on the initial norm $\|(u_0, v_0, \theta_0)\|_{H^1}$, in a continuous manner, and is independent of $\varepsilon$.
\end{proposition}

\begin{proof}
  Applying the operator $\nabla$ to equation (\ref{1.1}), multiplying the resultant by $-\nabla\Delta u$, and integrating over $\mathbb R^2$, then it follows from integration by parts, the H\"older and Cauchy-Schwarz inequalities that
  \begin{align}
\frac12\frac{d}{dt}&\|\Delta u\|_2^2+\|\nabla\Delta u\|_2^2
    =\int_{\mathbb R^2}\nabla[(u\cdot\nabla)u+\text{div}\,(v\otimes v)]:\nabla\Delta u dx\nonumber\\
    \leq&\int_{\mathbb R^2}(|u||\nabla^2u|+|\nabla u|^2+2|v||\nabla^2v|+2|\nabla v|^2)|\nabla\Delta u|dx\nonumber\\
    \leq&2(\|u\|_\infty\|\Delta u\|_2+\|v\|_\infty\|\Delta v\|_2+\|\nabla u\|_4^2+\|\nabla v\|_4^2)\|\nabla\Delta u\|_2\nonumber\\
    \leq&\frac12\|\nabla\Delta u\|_2^2+2(\|(u,v)\|_\infty^2\|(\Delta u,\Delta v)\|_2^2+\|(\nabla u,\nabla v)\|_4^4).\label{3.1}
  \end{align}
  By the elliptic estimate and the commutator estimate (Lemma \ref{commutator}), one has
  \begin{equation}
    \|\nabla^2(-\Delta)^{-1}\text{div}\, v\|_2+\|\nabla F\|_2\leq C(\|\nabla v\|_2+\|\nabla u\|_4\|\theta\|_4). \label{3.2}
  \end{equation}
  Applying the operator $\nabla$ to equation (\ref{2.6}), multiplying the resultant by $-\nabla\Delta w$, and integrating over $\mathbb R^2$, then it follows from integration by parts, (\ref{3.2}) and the H\"older and Young inequalities that
  \begin{align*}
    &\frac{1}{2}\frac{d}{dt}\|\Delta w\|_2^2+\|\nabla\Delta w\|_2^2 \\
    =&\int_{\mathbb R^2}\nabla\left[(u\cdot\nabla)w+(v\cdot\nabla)u+ \frac{1}{1-\varepsilon}(\nabla(-\Delta)^{-1}\text{div}\, v+F)\right]:\nabla\Delta wdx \nonumber\\
    \leq&[\|u\|_\infty\|\Delta w\|_2+\|\nabla u\|_4\|\nabla w\|_4+\|v\|_\infty\|\Delta u\|_2\nonumber\\
    &+\|\nabla v\|_4\|\nabla w\|_4+C(\|\nabla v\|_2+\|\nabla u\|_4\|\theta\|_4)] \|\nabla\Delta w\|_2\nonumber\\
    \leq&\frac12\|\nabla\Delta w\|_2^2+C(\|(u,v)\|_\infty^2\|(\Delta u,\Delta w)\|_2^2\nonumber\\
    &+\|(\nabla u,\nabla v,\nabla w,\theta)\|_4^4+\|\nabla v\|_2^2).
  \end{align*}
  Adding (\ref{3.1}) with the above inequality up, and multiplying the resultant by $t$, then
  \begin{align}
    &\frac{d}{dt}(t\|(\Delta u,\Delta w)\|_2^2)+t\|(\nabla\Delta u,\nabla\Delta w)\|_2^2\nonumber\\
    \leq&Ct\|(u, v)\|_\infty^2\|(\Delta u,\Delta v,\Delta w)\|_2^2 +C(t\|\nabla v\|_2^2\nonumber\\
    &+t\|(\nabla u,\nabla v,\nabla w,\theta)\|_4^4+\|(\Delta u,\Delta w)\|_2^2).\label{3.4}
  \end{align}

  Taking the $L^2(\mathbb R^2)$ inner product to equation (\ref{1.4}) with $-\Delta \theta$, then it follows from integration by parts and the Young inequality that
  \begin{align*}
    &\frac12\frac{d}{dt}\|\nabla\theta\|_2^2+\varepsilon\|\Delta\theta\|_2^2 =\int_{\mathbb R^2}(u\cdot\nabla \theta+\text{div}\,v)\cdot\Delta\theta dx\\
    =&-\int_{\mathbb R^2}(\partial_iu\cdot\nabla\theta\partial_i\theta +\nabla\text{div}\,v\cdot\nabla\theta)dx
    \leq \|\nabla u\|_\infty\|\nabla\theta\|_2^2+\frac12(\|\Delta v\|_2^2+\|\nabla\theta\|_2^2).
  \end{align*}
  The above inequality implies
  \begin{equation*}
    \frac{d}{dt}\|\nabla\theta\|_2^2+\varepsilon\|\Delta\theta\|_2^2\leq(2\|\nabla u\|_\infty+1)\|\nabla\theta\|_2^2+\|\Delta v\|_2^2,
  \end{equation*}
  which, combined with (\ref{3.4}), yields
  \begin{align}
    &\frac{d}{dt}(\|\nabla\theta\|_2^2+t\|(\Delta u,\Delta w)\|_2^2)  +t\|(\nabla\Delta u,\nabla\Delta w)\|_2^2+\varepsilon\|\Delta\theta\|_2^2\nonumber\\
    \leq&C(\|(u,v)\|_\infty^2+\|\nabla u\|_\infty+1)(\|\nabla\theta\|_2^2+t\|(\Delta u,\Delta v,\Delta w)\|_2^2)\nonumber\\
    &+C(t\|(\nabla u,\nabla v,\nabla w)\|_4^4+t\|\nabla v\|_2^2+\|(\Delta u,\Delta v,\Delta w)\|_2^2).\label{3.6}
  \end{align}

  Recalling the definitions of $\Phi$ and $w$, i.e. (\ref{2.1}) and (\ref{2.5}), one has
  \begin{equation}
    \label{3.7}
    \|\Delta v\|_2\leq\|\Delta w\|_2+2\|\nabla\theta\|_2,
  \end{equation}
  and furthermore, by the elliptic estimates, we have
  \begin{equation}
    \label{3.8}
    \|\nabla v\|_4\leq\|\nabla w\|_4+2\|\nabla\Phi\|_4\leq C(\|\nabla w\|_4+\|\theta\|_4).
  \end{equation}
  Thanks to (\ref{3.7}) and (\ref{3.8}), and denoting
  \begin{align*}
    A(t)=&\|\nabla\theta\|_2^2(t)+t\|(\Delta u,\Delta w)\|_2^2(t)+1,\\
    B(t)=&A(t)+t\|(\nabla\Delta u,\nabla\Delta w)\|_2^2(t)+\varepsilon\|\Delta\theta\|_2^2(t)+e,
  \end{align*}
  it follows from (\ref{3.6}) that
  \begin{equation}
    \label{3.9}
    A'(t)+B(t)\leq C(t+1)(\|(u,v)\|_\infty^2+\|\nabla u\|_\infty+1) A(t)+Cg(t),
  \end{equation}
  where
  $$
  g(t)=(t+1)(\|(\nabla u,\nabla w,\theta)\|_4^4+\|(\nabla v,\Delta u,\Delta w)\|_2^2).
  $$

  We need to evaluate the coefficients in front of A(t) in (\ref{3.9}), so that one can apply the logarithmic type Gronwall inequality (Lemma \ref{loggronwall}) to obtain the conclusion. First, by the Gagliardo-Nirenberg and Ladyzhenskaya inequalities, (\ref{3.8}) and the Young inequality, one can deduce
  \begin{align}
    \|(u,v)\|_\infty^2\leq& C(\|u\|_2\|\Delta u\|_2+\|v\|_2^{\frac23}\|\nabla v\|_4^{\frac43})\nonumber\\
    \leq&C\left[\|u\|_2\|\Delta u\|_2+\|v\|_2^{\frac23}(\|\nabla w\|_4+ \|\theta\|_4)^{\frac43}\right]\nonumber\\
    \leq&C\left[\|u\|_2\|\Delta u\|_2+\|v\|_2^{\frac23}(\|\nabla w\|_2^{\frac12}\|\Delta w\|_2^{\frac12}+ \|\theta\|_4)^{\frac43}\right]\nonumber\\
    \leq&C(\|(\Delta u,\Delta w,\nabla w,u, v)\|_2^2+\|\theta\|_4^2). \label{3.10}
  \end{align}
  Next, noticing that
  \begin{align*}
    \log(e+\|\nabla u\|_2^2+\|\nabla\Delta u\|_2^2)
    \leq&\log\left[\left(1+1/t\right)(1+\|\nabla u\|_2^2)(e+t\|\nabla\Delta u\|_2^2)\right]\\
    \leq&\log\left[\left(1+1/t\right)(1+\|\nabla u\|_2^2)\right] +\log (e+t\|\nabla\Delta u\|_2^2)\\
    \leq&\log\left[\left(1+1/t\right)(1+\|\nabla u\|_2^2)\right] +\log B(t),
  \end{align*}
  by the two-dimensional Br\'ezis--Gallouate--Wainger inequality of the form (see \cite{Brezis1,Brezis2})
  $$
  \|f\|_{L^\infty(\mathbb R^2)}\leq C(\|f\|_{H^1(\mathbb R^2)}+1)\log^{\frac12}(e+\|f\|_{H^2(\mathbb R^2)}),
  $$
  one has
  \begin{align}
    \|\nabla u\|_\infty\leq&C(1+\|\nabla u\|_{H^1})\log^{\frac12}(e+\|\nabla u\|_{H^2}^2)\nonumber\\
    \leq&C(1+\|\nabla u\|_{H^1})\log^{\frac12}(e+\|\nabla u\|_2^2+\|\nabla \Delta u\|_2^2)\nonumber\\
    \leq&C(1+\|\nabla u\|_{H^1})\left\{ \log^{\frac12}\left[\left(1+1/t\right)(1+\|\nabla u\|_2^2)\right] +\log^{\frac12} B(t)\right\}.\label{3.13}
  \end{align}

  Thanks to the estimates (\ref{3.10}) and (\ref{3.13}), it follows from (\ref{3.9}) and the Young inequality that
  \begin{align*}
    A'(t)+B(t)\leq& C(t+1)\left((1+\|\nabla u\|_{H^1})\log^{\frac12}B(t)+m(t)\right)A(t)+Cg(t)\nonumber\\
    \leq&C[(t+1)m(t)+(t+1)^2(1+\|\nabla u\|_{H^1}^2)+\log B(t)]A(t)+Cg(t),
  \end{align*}
  where the function $m$ is given by
  $$
  m(t)=\|(\Delta u,\Delta w,\nabla w,u,v)\|_2^2+\|\theta\|_4^2+(1+\|\nabla u\|_{H^1})\log^{\frac{1}{2}}\left[\left(1+1/t\right)(1+\|\nabla u\|_2^2)\right].
  $$
  By Propositions \ref{basicenergy}--\ref{secondenergy}, one can easily check that
  $(t+1)m(t)+(t+1)^2(1+\|\nabla u\|_{H^1}^2)$ and $g(t)$
  are integrable in $t$; moreover, for any positive time $T$, the $L^1((0,T))$ norm of each of them depends only on $T$ and the initial norm $\|(u_0, v_0, \theta_0)\|_{H^1}$. Therefore, by the logarithmic Gronwall inequality (Lemma \ref{loggronwall}), the conclusion follows.
\end{proof}

As a corollary of Propositions \ref{secondenergy}--\ref{secondtheta}, we can achieve the Lipschitz estimate on $u$, that is, the following:

\begin{corollary}\label{cor1}
Let $(u, v, \theta)$ be the unique global strong solution to system (\ref{1.1})--(\ref{1.4}), for $\varepsilon\in(0,\frac12)$, with initial data $(u_0, v_0, \theta_0)\in H^2(\mathbb R^2)$. Then, we have, for all $t \in [0,\infty)$,
  \begin{equation*}
    \int_0^t\|\nabla u\|_\infty ds\leq Ct^{\frac12}S_1^{\frac14}(t)S_2^{\frac14}(t),
  \end{equation*}
  for an absolute positive constant $C$, independent of $\varepsilon$, where $S_1(t)$ and $S_2(t)$ are the same functions as those in Proposition \ref{secondenergy} and Proposition \ref{secondtheta}, respectively.
\end{corollary}

\begin{proof}
  By the aid of the Gagliardo-Nirenberg and H\"older inequalities, it follows from Propositions \ref{secondenergy}--\ref{secondtheta} that
  \begin{align*}
  \int_0^t\|\nabla u\|_\infty ds\leq& C\int_0^t\|\nabla u\|_2^{\frac12}\|\nabla\Delta u\|_2^{\frac12}ds=C\int_0^t\|\nabla u\|_2^{\frac12}(s\|\nabla\Delta u\|_2^2)^{\frac14} s^{-\frac14}ds\\
  \leq&C\sup_{0\leq s\leq t}\|\nabla u\|_2^{\frac12}\left(\int_0^ts\|\nabla\Delta u\|_2^2ds\right)^{\frac14} \left(\int_0^ts^{-\frac13}ds\right)^{\frac34}\\
  \leq&Ct^{\frac12}\sup_{0\leq s\leq t}\|\nabla u\|_2^{\frac12}\left(\int_0^ts\|\nabla\Delta u\|_2^2ds\right)^{\frac14}
  \leq Ct^{\frac12}S_1^{\frac14}(t)S_2^{\frac14}(t),
  \end{align*}
  proving the conclusion.
\end{proof}

We will prove the $L^\infty$ estimate on $\theta$.
To this end, we first introduce another auxiliary function
\begin{equation}
  \phi:=\text{div}\,v-\frac{\theta}{1-\varepsilon}. \label{4.1}
\end{equation}
We would like to point out that $\phi$ looks like the effective viscous flux for the compressible Navier-Stokes equations. Applying the operator ${div}$ to equation (\ref{1.3}) yields
\begin{equation*}
  \partial_t\text{div}\, v+u\cdot\nabla\text{div}\,v-\Delta(\text{div}\, v-\theta)+2\partial_iu\cdot\nabla v^i=0.
\end{equation*}
Multiplying equation (\ref{1.4}) by $\frac{1}{1-\varepsilon}$, and subtracting the resultant from the above equation,
one obtains the following equation for $\phi$
\begin{equation}
  \partial_t\phi+u\cdot\nabla\phi-\Delta\phi+2\partial_iu\cdot\nabla v^i-\frac{\text{div}\, v}{1-\varepsilon}=0. \label{4.2}
\end{equation}

We state and prove the required estimates on $\phi$ in the following proposition.

\begin{proposition}  \label{estphi}
Let $(u, v, \theta)$ be the unique global strong solution to system (\ref{1.1})--(\ref{1.4}), with initial data $(u_0, v_0, \theta_0)\in H^2(\mathbb R^2)$. Let $\varepsilon\in(0,\frac12)$ and $\phi$ be given by (\ref{4.1}). Then, we have the following estimate, for all $t \in [0,\infty)$,
  \begin{equation*}
  \int_0^t \|\phi\|_\infty(s) ds\leq S_3(t),
  \end{equation*}
  where $S_3(t)$ is an explicit nondecreasing continuous function on $[0,\infty)$, which depends only on the initial norm $\|(u_0, v_0,\theta_0)\|_{H^1}$, in a continuous manner, and is independent of $\varepsilon$.
\end{proposition}

\begin{proof}
  Taking the $L^2(\mathbb R^2)$ product to equation (\ref{4.2}) with $-\Delta\phi$, then it follows from the H\"older, Ladyzhenskaya and Young inequalities that
  \begin{align*}
\frac12\frac{d}{dt}\|\nabla \phi\|_2^2+\|\Delta\phi\|_2^2
    =& \int_{\mathbb R^2}(u\cdot\nabla\phi+2\partial_iu\cdot\nabla v^i-(1-\varepsilon)^{-1}\text{div}\, v)\Delta\phi dx\\
    \leq&2(\|u\|_4\|\nabla\phi\|_4+\|\nabla u\|_4\|\nabla v\|_4+\|\nabla v\|_2)\|\Delta\phi\|_2\\
    \leq&C(\|u\|_4\|\nabla\phi\|_2^{\frac12}\|\Delta\phi\|_2^{\frac12}+\|\nabla u\|_4\|\nabla v\|_4+\|\nabla v\|_2)\|\Delta\phi\|_2\\
    \leq&\frac12\|\Delta\phi\|_2^2+C(\|u\|_4^4\|\nabla\phi\|_2^2+\|(\nabla u,\nabla v)\|_4^4+\|\nabla v\|_2^2),
  \end{align*}
  and thus
  \begin{equation*}
    \frac{d}{dt}\|\nabla\phi\|_2^2+\|\Delta\phi\|_2^2\leq C(\|u\|_4^4\|\nabla\phi\|_2^2+\|(\nabla u,\nabla v)\|_4^4+\|\nabla v\|_2^2).
  \end{equation*}
  Multiplying both sides of the above inequality by $t$, one has
   \begin{align}
    \frac{d}{dt}(t\|\nabla\phi\|_2^2)+t\|\Delta\phi\|_2^2
    \leq  C(t\|u\|_4^4+1)\|\nabla\phi\|_2^2+Ct(\|(\nabla u,\nabla v)\|_4^4+\|\nabla v\|_2^2).\label{4.3}
  \end{align}
  Recalling the definitions of $\Phi, w$ and $\phi$, i.e., (\ref{2.1}), (\ref{2.5}) and (\ref{4.1}), it follows from the elliptic estimate that
  \begin{equation}
  \|\phi\|_2^2= \|\text{div}\,w-(1-\varepsilon)^{-1}\text{div}\, \Phi-(1-\varepsilon)^{-1}\theta\|_2^2\leq C(\|\nabla w\|_2^2+\|\theta\|_2^2),\label{4.4}
  \end{equation}
  and
  \begin{equation*}
    \|\nabla\phi\|_2^2= \|\nabla(\text{div}\, w-(1-\varepsilon)^{-1}\text{div}\,\Phi-(1-\varepsilon)^{-1} \theta)\|_2^2
    \leq C(\|\Delta w\|_2^2+\|\nabla\theta\|_2^2).
  \end{equation*}
  Thanks to the above estimates, and recalling (\ref{3.8}), it follows from (\ref{4.3}) and the Ladyzhenskaya inequality that
  \begin{align*}
    \frac{d}{dt}(t\|\nabla\phi\|_2^2)+t\|\Delta\phi\|_2^2
    \leq&C(t\|u\|_4^4+1)\|(\Delta w,\nabla\theta)\|_2^2+Ct(\|(\nabla u,\nabla w,\theta)\|_4^4+\|\nabla v\|_2^2)\nonumber\\
    \leq&C(t\|u\|_2^2\|\nabla u\|_2^2+1)\|(\Delta w,\nabla\theta)\|_2^2+Ct(\|\nabla v\|_2^2\nonumber\\
    &+\|\theta\|_4^4+\|(\nabla u,\nabla w)\|_2^2\|(\Delta u, \Delta w)\|_2^2),
  \end{align*}
  from which, integrating in $t$ and using Propositions \ref{basicenergy}--\ref{secondtheta}, one arrives at
  \begin{equation}
    \sup_{0\leq s\leq t}(s\|\nabla\phi\|_2^2(s))+\int_0^ts\|\Delta\phi\|_2^2(s)ds\leq Q_1(t), \label{4.5}
  \end{equation}
  where $Q_1(t)$ is a continuously increasing function on $[0,\infty)$, which depends only on the initial norm $\|(u_0, v_0, \theta_0)\|_{H^1}$. Therefore, recalling (\ref{4.4}), it follows from the Gagliardo-Nirenberg and H\"older inequalities that
  \begin{align*}
\int_0^t\|\phi\|_\infty ds\leq& C\int_0^t\|\phi\|_2^{\frac12}\|\Delta\phi\|_2^{\frac12}ds =C\int_0^t\|\phi \|_2^{\frac12}(s\|\Delta\phi\|_2^2)^{\frac14}s^{-\frac14}ds\nonumber\\
    \leq&C\sup_{0\leq s\leq t}\|\phi\|_2^{\frac12}\left(\int_0^ts\|\Delta\phi\|_2^2ds\right)^{\frac14} \left(\int_0^ts^{-\frac13}ds\right)^{\frac34}\nonumber\\
    \leq&Ct^{\frac12}Q_1^{\frac14}(t)\sup_{0\leq s\leq t}(\|\nabla w\|_2+\|\theta\|_2)^{\frac12}\nonumber\\
    \leq& Ct^{\frac12}Q_1^{\frac14}(t)(S_1^{\frac14}(t)+\|(u_0, v_0, \theta_0)\|_2^\frac12)=:S_{3}(t),
  \end{align*}
  proving the conclusion.
\end{proof}

We can now get the $L^\infty$ bound of $\theta$.

\begin{proposition}[$L^\infty$ of the temperature]
  \label{supertheta}
  Let $(u, v,\theta)$ be the unique strong solution to system (\ref{1.1})--(\ref{1.4}), $\varepsilon\in(0,\frac12)$, with initial data $(u_0, v_0, \theta_0)\in H^2(\mathbb R^2)$. Then, we have the following estimate, for all $t \in [0,\infty)$,
  $$
  \sup_{0\leq s\leq t}\|\theta\|_\infty^2(s)\leq S_4(t),
  $$
  where $S_4(t)$ is a explicit nondecreasing continuous function on $[0,\infty)$, which depends only on the initial norm $\|(u_0,v_0,\theta_0)\|_{H^1}+\|\theta_0\|_\infty$, in the continuous manner, and is independent of $\varepsilon$.
\end{proposition}

\begin{proof}
  Note that equation (\ref{1.4}) can be rewritten as
  \begin{equation*}
    \partial_t\theta+u\cdot\nabla\theta-\varepsilon\Delta \theta+\frac{\theta}{1-\varepsilon} +\phi=0.
  \end{equation*}
  Introducing $\Theta=e^{\frac{t}{1-\varepsilon}}\theta,$
  then it satisfies
  $$
  \partial_t\Theta+u\cdot\nabla\Theta-\varepsilon\Delta\Theta+e^{\frac{t} {1-\varepsilon}}\phi=0.
  $$
  Define another function $\Lambda$ as
  $$
  \Lambda=\Theta-\|\theta_0\|_\infty-\int_0^te^{\frac{s}{1-\varepsilon}}\|
  \phi\|_\infty(s)ds,
  $$
  then one obtains
  \begin{align*}
    \partial_t\Lambda+u\cdot\nabla\Lambda-\varepsilon\Delta\Lambda =&\partial_t\Theta-e^{\frac{t}{1-\varepsilon}}\|\phi\|_\infty(t) +u\cdot\nabla\Theta-\varepsilon\Delta\Theta\\
    =&-e^{\frac{t}{1-\varepsilon}}\|\phi\|_\infty(t) -e^{\frac{t}{1-\varepsilon}}\phi(x,t)\leq0.
  \end{align*}
  Multiplying the above equation by $\Lambda^+=\max\{\Lambda,0\}$, integrating the resultant over $\mathbb R^2$, then it follows from integration by parts that
  \begin{align*}
    \frac12\frac{d}{dt}\|\Lambda^+\|_2^2+\varepsilon\|\nabla\Lambda^+\|_2^2 \leq0,
  \end{align*}
  from which, noticing that $\Lambda^+|_{t=0}=0$, one obtains
  $\|\Lambda^+\|_2^2(t)\leq\|\Lambda^+\|_2^2(0)=0$, and thus $\Lambda^+=0$. Therefore, recalling the definitions of $\Theta$ and $\Lambda$, we have
  $$
  e^{\frac{t}{1-\varepsilon}}\theta-\|\theta_0\|_\infty-\int_0^te^{\frac{s}{ 1-\varepsilon}}\|\phi\|_\infty(s)ds\leq 0,
  $$
  which implies
  $$
  \theta\leq e^{-\frac{t}{1-\varepsilon}}\left(\|\theta_0\|_\infty+\int_0^te^{\frac{s}{ 1-\varepsilon}}\|\phi\|_\infty(s)ds\right)\leq  \|\theta_0\|_\infty+\int_0^t\|\phi\|_\infty(s)ds.
  $$
  Similarly, by considering $-\theta$, the same argument as above yields $$
  -\theta\leq \|\theta_0\|_\infty+\int_0^t\|\phi\|_\infty(s)ds.
  $$
  Combining the above two inequalities and using Proposition \ref{estphi}, we obtain the conclusion.
\end{proof}

On account of Propositions \ref{basicenergy}--\ref{supertheta} and Corollary \ref{cor1}, we can get the a priori estimates stated in the following corollary.

\begin{corollary}
  \label{cor2}
  Let $(u, v, \theta)$ be the unique global strong solution to system (\ref{1.1})--(\ref{1.4}), $\varepsilon\in(0,\frac12)$, with initial data $(u_0, v_0, \theta_0)\in H^2(\mathbb R^2)$. Then, we have, for all $t \in [0,\infty)$,
  \begin{align*}
    \sup_{0\leq s\leq t}(\|(u, v, \theta)\|_{H^1}^2+\|\theta\|_\infty^2)+\int_0^t(\|(\Delta u, \Delta v,\partial_tu, \partial_tv,\partial_t\theta)\|_2^2+\|\nabla u\|_\infty)ds\leq S(t),
  \end{align*}
  where $S$ is an explicit nondecreasing continuous function on $[0,\infty)$, depending only on the initial norm $\|(u_0, v_0, \theta_0)\|_{H^1}+\|\theta_0\|_\infty$, in the continuous manner, and is independent of $\varepsilon$.
\end{corollary}

\begin{proof}
  Recalling the definitions of $\Phi$ and $w$, i.e. (\ref{2.1}) and (\ref{2.5}), it follows from the elliptic estimates that
  \begin{eqnarray*}
    &&\|\Delta v\|_2\leq\|\Delta w\|_2+2\|\Delta\Phi\|_2=\|\Delta w\|_2+ 2\|\nabla\theta\|_2,\\
    &&\|\nabla v\|_2\leq\|\nabla w\|_2+2\|\nabla\Phi\|_2\leq\|\nabla w\|_2 +C\|\theta\|_2.
  \end{eqnarray*}
  Using (\ref{1.3}), it follows from the H\"older, Ladyzhenskaya and Young inequality that
  \begin{align*}
    \|\partial_tv\|_2^2\leq&C(\|\Delta v\|_2^2+\|\nabla\theta\|_2^2+\|( u\cdot\nabla)v\|_2^2+\|(v\cdot\nabla)u\|_2^2)\\
    \leq&C(\|\Delta v\|_2^2+\|\nabla\theta\|_2^2+\|u\|_4^2\|\nabla v\|_4^2+\|v\|_4^2\|\nabla u\|_4^2)\\
    \leq&C(\|\Delta v\|_2^2+\|\nabla\theta\|_2^2+\|u\|_2\|\nabla u\|_2\|\nabla v\|_2\|\Delta v\|_2+\|v\|_2\|\nabla v\|_2 \|\nabla u\|_2\|\Delta u\|_2)\\
    \leq&C[\|\Delta u\|_2^2+\|\Delta v\|_2^2+\|\nabla\theta\|_2^2+(\|u\|_2^2+\|v\|_2^2)\|\nabla u\|_2^2\|\nabla v\|_2^2].
  \end{align*}
  Therefore, by Propositions \ref{basicenergy}--\ref{secondtheta}, we have
  \begin{align}
    &\sup_{0\leq s\leq t}\|\nabla v\|_2^2+\int_0^t(\|\Delta v\|_2^2+\|\partial_t v\|_2^2)ds\nonumber\\
    \leq&C\sup_{0\leq s\leq t}\|(\nabla w,\theta)\|_2^2+C\int_0^t (\|(\Delta u,\Delta w,\nabla\theta)\|_2^2+\|(u, v)\|_2^2\|\nabla u\|_2^2\|\nabla v\|_2^2)ds\nonumber\\
    \leq&C\sup_{0\leq s\leq t}\|(\nabla w,\theta)\|_2^2+C\int_0^t \|(\Delta u,\Delta w,\nabla\theta)\|_2^2ds\nonumber\\
    &+C\left(\sup_{0\leq s\leq t}\|(u, v)\|_2^2\|\nabla u\|_2^2\right)\int_0^t\|\nabla v\|_2^2ds\nonumber\\
    \leq&C(\|(u_0, v_0, \theta_0)\|_2^2+S_1(t))+CtS_2(t)+C\|(u_0, v_0,\theta_0)\|_2^4S_1(t). \label{A1}
  \end{align}
  By equation (\ref{1.4}), and using the Gagliardo-Nirenberg inequality, one has
  \begin{align*}
    \|\partial_t\theta\|_2^2\leq&C(\|u\cdot\nabla\theta\|_2^2 +\varepsilon^2\|\Delta\theta\|_2^2+\|\nabla v\|_2^2)\leq C(\|u\|_\infty^2\|\nabla\theta\|_2^2+\varepsilon^2\|\Delta\theta\|_2^2 +\|\nabla v\|_2^2)\\
    \leq&C(\|u\|_2\|\Delta u\|_2\|\nabla\theta\|_2^2+\varepsilon^2\|\Delta\theta\|_2^2 +\|\nabla v\|_2^2).
  \end{align*}
  Thus, it follows from Propositions \ref{basicenergy}--\ref{secondtheta} and the H\"older inequality that
  \begin{align}
    \int_0^t\|\partial_t\theta\|_2^2ds\leq&C\sup_{0\leq s\leq t}\|\nabla\theta\|_2^2\int_0^t\|u\|_2\|\Delta u\|_2 ds+C\int_0^t( \varepsilon^2\|\Delta\theta\|_2^2+\|\nabla v\|_2^2)ds\nonumber\\
    \leq&CS_2(t)\left(\int_0^t\|u\|_2^2ds\right)^{\frac12}\left(\int_0^t \|\Delta u\|_2^2ds\right)^{\frac12}+C(t\|(u_0, v_0, \theta_0)\|_2^2+S_2(t))\nonumber\\
    \leq&CS_2(t)S_1^{\frac12}(t)t^{\frac12}\|(u_0, v_0,\theta_0)\|_2+C(t \|(u_0, v_0, \theta_0)\|_2^2+S_2(t)). \label{A2}
  \end{align}
  Combining (\ref{A1}) with (\ref{A2}), by Propositions \ref{basicenergy}--\ref{supertheta} and Corollary \ref{cor1}, the conclusion follows.
\end{proof}

With the a priori estimate, i.e. Corollary \ref{cor2}, we can now give the proof of global existence of strong solutions to system (\ref{main1.1})--(\ref{main1.4}), with $H^1$ initial data as follows.

\begin{proof}[\textbf{Proof of the global existence part of Theorem \ref{theorem}}] Choose a sequence of initial data $(u_{0\varepsilon}, v_{0\varepsilon}, \theta_{0\varepsilon})\in H^2(\mathbb R^2)$, with $\text{div}\, u_{0\varepsilon}=0$, such that
\begin{eqnarray*}
  &(u_{0\varepsilon}, v_{0\varepsilon}, \theta_{0\varepsilon})\rightarrow(u_0, v_0, \theta_0),\quad\mbox{ in }H^1(\mathbb R^2),\mbox{ as }\varepsilon\rightarrow0, \\
  &\|\theta_{0\varepsilon}\|_\infty\leq\|\theta_0\|_\infty\mbox{ and } \|(u_{0\varepsilon}, v_{0\varepsilon}, \theta_{0\varepsilon})\|_{H^1}\leq\|(u_0, v_0, \theta_0)\|_{H^1}.
\end{eqnarray*}

By Proposition \ref{glob} and Corollary \ref{cor2}, for each $\varepsilon\in(0,\frac12)$, there is a unique global strong solution $(u_\varepsilon, v_\varepsilon, \theta_\varepsilon)$ to system (\ref{1.1})--(\ref{1.4}), with initial data $(u_{0\varepsilon}, v_{0\varepsilon}, \theta_{0\varepsilon})$, satisfying
\begin{align}
  \sup_{0\leq t\leq T}&(\|(u_\varepsilon, v_\varepsilon, \theta_\varepsilon)\|_{H^1}^2+\|\theta_\varepsilon\|_\infty^2) +\int_0^T(\|\nabla u_\varepsilon\|_\infty\nonumber\\
  &+\|(\Delta u_\varepsilon, \Delta v_\varepsilon,\partial_tu_\varepsilon,\partial_tv_\varepsilon, \partial_t\theta_\varepsilon)\|_2^2)dt\leq S(T),\label{E0}
\end{align}
for a positive constant $S$ depending only on $T$ and $\|(u_0, v_0, \theta_0)\|_{H^1}+\|\theta_0\|_\infty$, in a continuous manner, and is independent of $\varepsilon$.

Thanks to the above estimates, there is a subsequence depending on $T$, still denoted by $(u_\varepsilon, v_\varepsilon, \theta_\varepsilon)$, and $(u, v, \theta)$, such that
\begin{eqnarray}
  &(u_\varepsilon, v_\varepsilon, \theta_\varepsilon)\overset{*}{\rightharpoonup}(u, v, \theta),\quad\mbox{in }L^\infty(0,T;H^1(\mathbb R^2)),\label{E1}\\
  &(u_\varepsilon, v_\varepsilon)\rightharpoonup(u, v),\quad\mbox{in }L^2(0,T;H^2(\mathbb R^2)),\label{E2}\\
  &(\partial_tu_\varepsilon,\partial_tv_\varepsilon, \partial_t\theta_\varepsilon)\rightharpoonup(\partial_tu, \partial_tv,\partial_t\theta),\quad\mbox{in }L^2(0,T; L^2(\mathbb R^2)),\label{E3}
\end{eqnarray}
where $\rightharpoonup$ and $\overset{*}{\rightharpoonup}$ denote the weak and weak-* convergences, respectively.
Moreover, for any positive integer $k$, by the Aubin-Lions lemma, there is a subsequence depending on $T$ and $k$, still denoted by $(u_\varepsilon, v_\varepsilon,\theta_\varepsilon)$, such that
$$
(u_\varepsilon, v_\varepsilon,\theta_\varepsilon)\rightarrow(u, v,\theta),\quad\mbox{in }L^2(0,T;H^1(B_k))\cap C([0,T];L^2(B_k)),
$$
where $B_k$ denotes the ball in $\mathbb R^2$ of radius $k$ and centered at the origin. Using the Cantor diagonal argument in $\varepsilon$ and $k$, there is a subsequence depending on $T$, still denoted by $(u_\varepsilon, v_\varepsilon, \theta_\varepsilon)$, such that
\begin{equation}
  (u_\varepsilon, v_\varepsilon,\theta_\varepsilon)\rightarrow(u, v, \theta),\quad\mbox{in }L^2(0,T;H^1(B_R))\cap C([0,T];L^2(B_R)), \label{E4}
\end{equation}
for any positive $R$.

Note that the convergent subsequences in (\ref{E1})--(\ref{E4}) may depend on $T$. However, by choosing $\{T_m\}_{m=1}^\infty$, with $T_m\rightarrow\infty$, one can use again the Cantor diagonal argument in $m$ and $\varepsilon$ to show that the subsequences in (\ref{E1})--(\ref{E4}) can be chosen independent of $T$. Therefore, without loss of generality, we assume that (\ref{E1})--(\ref{E4}) is satisfied for any positive $T$ and $R$. Thanks to these convergences, one can take the limit $\varepsilon\rightarrow0$,i.e.\, for the corresponding subsequence of $\varepsilon$, to show that $(u, v, \theta)$ is a global strong solution to system (\ref{main1.1})--(\ref{main1.4}), with initial data $(u_0, v_0,\theta_0)$. Moreover, by the weakly lower semi-continuity of the norms, we have the following regularities on $(u, v, \theta)$:
\begin{eqnarray*}
  &&(u, v, \theta)\in L^\infty(0,T;H^1(\mathbb R^2)),\quad(u, v)\in L^2(0,T;H^2(\mathbb R^2)),\\
  &&(\partial_tu,\partial_tv,\partial_t\theta)\in L^2(0,T;L^2(\mathbb R^2)),
\end{eqnarray*}
for every $T>0$. 
Noticing that
\begin{eqnarray*}
  &&X:=\{f\in L^2(0,T;H^2(\mathbb R^2))|\partial_tf \in L^2(0,T;L^2(\mathbb R^2))\}\hookrightarrow C([0,T]; H^1(\mathbb R^2)),\\
  &&Y:=\{f\in L^2(0,T;H^1(\mathbb R^2))|\partial_tf\in L^2(0,T;L^2(\mathbb R^2))\}\hookrightarrow C([0,T]; L^2(\mathbb R^2)),
\end{eqnarray*}
the previous regularities of $(u, v,\theta)$ imply
$$
(u, v)\in C([0,T];H^1(\mathbb R^2))\mbox{ and }\theta\in C([0,T];L^2(\mathbb R^2)).
$$

To complete the proof of the existence, we still need to show the regularities that $\theta\in L^\infty(0,T; L^\infty(\mathbb R^2))$ and $\nabla u\in L^1(0,T; L^\infty(\mathbb R^2))$. Note that (\ref{E4}) implies that there is a subsequence, still denoted by $(u_\varepsilon, v_\varepsilon,\theta_\varepsilon)$, such that
$$
(\theta_\varepsilon, \nabla u_\varepsilon)\rightarrow(\theta, \nabla u),\quad\mbox{a.e. in }\mathbb R^2\times(0,T).
$$
Thanks to the above pointwise convergence, and recalling (\ref{E0}), it follows that
$$
\|\theta\|_{L^\infty(\mathbb R^2\times(0,T))}=\sup_{(x,t)\in\mathbb R^2\times(0,T)}\lim_{\varepsilon\rightarrow0}|\theta_\varepsilon(x,t)|\leq S(T),
$$
for every $T>0$.
Furthermore, by the Fatou lemma, we have
\begin{align*}
  \int_0^T\|\nabla u\|_\infty dt=&\int_0^T\sup_{x\in\mathbb R^2}|\nabla u(x, t)|dt=\int_0^t\sup_{x\in\mathbb R^2}\lim_{\varepsilon\rightarrow0}|\nabla u_\varepsilon(x,t)|dt \leq\\
  \leq&\int_0^T\varliminf_{\varepsilon\rightarrow0}\sup_{x\in\mathbb R^2}|\nabla u_\varepsilon(x,t)|dt=\int_0^t\varliminf_{\varepsilon\rightarrow 0}\|\nabla u_\varepsilon\|_\infty dt \\
  \leq&\varliminf_{\varepsilon\rightarrow0}\int_0^T\|\nabla u_\varepsilon\|_\infty dt\leq S(T)
\end{align*}
for every $T>0$.
Therefore, $(u, v, \theta)$ has the regularities stated in Theorem \ref{theorem}. This completes the proof of the existence part of Theorem \ref{theorem}.
\end{proof}

\section{Uniqueness of solutions}
\label{secuniq}
In this section, we give the proof of the uniqueness of strong solutions. Here we adopt an idea that was introduce in \cite{LarLunTiti} and \cite{LITITIXIN}. 

\begin{proof}[\textbf{Proof of the uniqueness part of Theorem \ref{theorem}}]
Let $(u_i, v_i, \theta_i)$, $i=1,2$, be two strong solutions to system (\ref{main1.1})--(\ref{main1.4}), with the same initial data $(u_{0}, v_{0}, \theta_{0})$. Denote
$$
(u, v, \theta)=(u_1-u_2, v_1-v_2, \theta_1-\theta_2),
 $$
and define the functions
 $$
 (\xi,\eta,\zeta)=(I-\Delta)^{-1}(u,v,\theta),
 $$
 in other words, $(\xi,\eta,\zeta)$ is the unique solution to
 $$
 (\xi-\Delta\xi,\eta-\Delta\eta,\zeta-\Delta\zeta)=(u, v, \theta),\quad (\xi,\eta,\zeta)\rightarrow0,\mbox{ as }|x|\rightarrow\infty.
 $$
 Then $(u, v, \theta)$ satisfies the system
\begin{eqnarray}
  &&\partial_tu-\Delta u+\nabla p=-\text{div}\,(u_1\otimes u+u\otimes u_2+v_1\otimes v+v\otimes v_2),\label{5.1}\\
  &&\text{div}\,u=0,\label{5.2}\\
  &&\partial_tv-\Delta v+\nabla\theta=-\text{div}\,(v_1\otimes u+v\otimes u_2)-(v_1\cdot\nabla)u-(v\cdot\nabla)u_2,\label{5.3}\\
  &&\partial_t\theta=-\text{div}\,(u_1\theta +u\theta_2)-\text{div}\,v.\label{5.4}
\end{eqnarray}
Recalling that
$$
(u,v,\theta)\in L^2(0,T; H^1(\mathbb R^2)),\quad(\partial_tu,\partial_tv,\partial_t\theta)\in L^2(0,T;L^2(\mathbb R^2)),
$$
we have
$$
(\xi,\eta,\zeta)\in L^2(0,T; H^3(\mathbb R^2)),\quad(\partial_t\xi,\partial_t\eta,\partial_t\zeta)\in L^2(0,T; H^2(\mathbb R^2)).
$$

Applying the operator $(I-\Delta)^{-1}$ to equation (\ref{5.1}) yields
$$
\partial_t\xi-\Delta\xi+\nabla\tilde p=-(I-\Delta)^{-1}\text{div}\,(u_1\otimes u+u\otimes u_2+v_1\otimes v+v\otimes v_2),
$$
where $\tilde p=(I-\Delta)^{-1}p$. Taking the $L^2(\mathbb R^2)$ product to the above equation with $\xi-\Delta\xi$, and noticing that $\text{div}\,\xi=0$, then it follows from integration by parts, the H\"older, Ladyzhenskaya and Young inequalities that
\begin{align}
  &\frac12\frac{d}{dt}(\|\xi\|_2^2+\|\nabla\xi\|_2^2)+\|\nabla\xi\|_2^2+ \|\Delta\xi\|_2^2\nonumber\\
  =&\int_{\mathbb R^2}(u_1\otimes u+u\otimes u_2+v_1\otimes v+v\otimes v_2):\nabla\xi dx\nonumber\\
  \leq&\int_{\mathbb R^2}[(|u_1|+|u_2|)(|\xi|+|\Delta\xi|)+(|v_1|+|v_2|)(|\eta|+|\Delta\eta|)] |\nabla\xi|dx\nonumber\\
  \leq&(\|u_1\|_4+\|u_2\|_4)(\|\xi\|_4\|\nabla\xi\|_2 +\|\Delta\xi\|_2\|\nabla\xi\|_4 )\nonumber\\
  &+(\|v_1\|_4+\|v_2\|_4)(\|\eta\|_4\|\nabla \xi\|_2+\|\Delta\eta\|_2\|\nabla\xi\|_4)\nonumber\\
  \leq&C(\|u_1\|_4+\|u_2\|_4)(\|\xi\|_2^{\frac12}\|\nabla\xi\|_2^{\frac32} +\|\Delta\xi\|_2^{\frac32}\|\nabla\xi\|_2^{\frac12})\nonumber\\
  &+C(\|v_1\|_4+\|v_2\|_4)(\|\eta\|_2^{\frac12}\|\nabla\eta\|_2^{\frac12} \|\nabla\xi\|_2+\|\Delta\eta\|_2\|\Delta\xi\|_2^{\frac12} \|\nabla\xi\|_2^{\frac12})\nonumber\\
  \leq&C(\|u_1\|_4^4+\|u_2\|_4^4+\|v_1\|_4^4+\|v_2\|_4^4)(\|\nabla\xi\|_2^2+ \|\xi\|_2^2+\|\eta\|_2^2)\nonumber\\
  &+\frac16(\|\Delta\xi\|_2^2+\|\nabla\xi\|_2^2 +\|\Delta\eta\|_2^2+\|\nabla\eta\|_2^2).\label{5.5}
\end{align}

Applying the operator $(I-\Delta)^{-1}$ to equation (\ref{5.3}) yields
$$
\partial_t\eta-\Delta\eta+\nabla\zeta=-(I-\Delta)^{-1}[\text{div}\,(v_1\otimes u+v\otimes u_2)+(v_1\cdot\nabla)u-(v\cdot\nabla)u_2].
$$
Taking the $L^2(\mathbb R^2)$ inner product to the above equations with $(I-\Delta)\eta$, then it follows from integration by parts that
\begin{align}
  &\frac12\frac{d}{dt}(\|\nabla\eta\|_2^2+\|\eta\|_2^2)+\|\Delta\eta\|_2^2 +\|\nabla \eta\|_2^2\nonumber\\
  =&-\int_{\mathbb R^2}\{[\text{div}\,(v_1\otimes u+v\otimes u_2)+(v_1\cdot\nabla)u\nonumber\\
  &-(v\cdot\nabla)u_2]\cdot\eta+\nabla\zeta\cdot (\eta-\Delta\eta)\} dx\nonumber\\
  =&\int_{\mathbb R^2}[(v_1\otimes u+v\otimes u_2):\nabla\eta+\text{div}\,v_1u\cdot\eta+(v_1\cdot\nabla)\eta\cdot u\nonumber\\
  &-(v\cdot\nabla u_2)\cdot\eta+\nabla\zeta\cdot\Delta\eta-\nabla\zeta\cdot\eta]dx\nonumber\\
  \leq&\int_{\mathbb R^2}(2|v_1||u||\nabla\eta|+|u_2||v||\nabla\eta|+|\nabla v_1||u||\eta|\nonumber\\
  &+|\nabla u_2||v||\eta|+|\nabla\zeta||\Delta\eta|+|\nabla\zeta||\eta|)dx.\label{5.6}
\end{align}
We estimate the quantities on the right-hand side of (\ref{5.6}) as follows. By the H\"older, Ladyzhenskaya and Young inequalities, we have
\begin{align*}
  &\int_{\mathbb R^2}(2|v_1||u||\nabla\eta|+|u_2||v||\nabla\eta|)dx\\
  \leq&2\int_{\mathbb R^2}[|v_1|(|\Delta\xi|+|\xi|)|\nabla\eta|+|u_2|(|\Delta\eta|+|\eta|)|\nabla\eta|]dx\\
  \leq&2\|v_1\|_4(\|\Delta\xi\|_2\|\nabla\eta\|_4+\|\xi\|_4\|\nabla\eta\|_2) +2\|u_2\|_4(\|\Delta\eta\|_2\|\nabla\eta\|_4+ \|\eta\|_4\|\nabla\eta\|_2)\\
  \leq&C\|v_1\|_4(\|\Delta\xi\|_2\|\nabla\eta\|_2^{\frac12}\|\Delta\eta\|_2^{\frac12} +\|\xi\|_2^{\frac12}\|\nabla\xi\|_2^{\frac12}\|\nabla\eta\|_2)\\
  &+C\|u_2\|_4(\|\Delta\eta\|_2^{\frac32}\|\nabla\eta\|_2^{\frac12}+ \|\eta\|_2^{\frac12}\|\nabla\eta\|_2^{\frac32})\\
  \leq&\frac{1}{18}(\|\Delta\xi\|_2^2+\|\Delta\eta\|_2^2+\|\nabla\eta\|_2^2+ \|\nabla\xi\|_2^2)+C(\|v_1\|_4^4\|\nabla\eta\|_2^2+ \\
  &+\|v_1\|_4^4\|\xi\|_2^2+\|u_2\|_4^4\|\nabla\eta\|_2^2+\|u_2\|_4^4\|\eta \|_2^2)\\
  \leq&\frac{1}{18}\|(\Delta\xi,\Delta\eta,\nabla\eta,\nabla\xi)\|_2^2+C\|(v_1, u_2)\|_4^4\|(\nabla\eta,\xi,\eta)\|_2^2.
  \end{align*}
  By the H\"older, Ladyzhenskaya, Gagliardo-Nirenberg and Young inequality, one has
  \begin{align*}
  &\int_{\mathbb R^2}(|\nabla v_1||u||\eta|+|\nabla u_2||v||\eta|)dx\\
  \leq&\int_{\mathbb R^2}[|\nabla v_1|(|\Delta\xi|+|\xi|)|\eta|+|\nabla u_2|(|\Delta\eta|+|\eta|)|\eta|]dx\\
  \leq&\|\nabla v_1\|_2(\|\Delta\xi\|_2\|\eta\|_\infty+\|\xi\|_4\|\eta\|_4)+\|\nabla u_2\|_2(\|\Delta\eta\|_2\|\eta\|_\infty+\|\eta\|_4^2)\\
  \leq&C\|\nabla v_1\|_2(\|\Delta\xi\|_2\|\eta\|_2^{\frac12}\|\Delta\eta\|_2^{\frac12} +\|\xi\|_2^{\frac12}\|\nabla\xi\|_2^{\frac12}\|\eta\|_2^{\frac12} \|\nabla\eta\|_2^{\frac12})\\
  &+C\|\nabla u_2\|_2(\|\Delta\eta\|_2\|\eta\|_2^{\frac12}\|\Delta\eta\|_2^{\frac12} +\|\eta\|_2\|\nabla\eta\|_2)\\
  \leq&\frac{1}{18}\|(\Delta\xi,\Delta\eta,\nabla\xi,\nabla\eta)\|_2^2 +C(1+\|(\nabla v_1,\nabla u_2)\|_2^4)\|(\eta,\xi)\|_2^2,
\end{align*}
and
$$
\int_{\mathbb R^2}(|\nabla\zeta||\Delta\eta|+|\nabla\zeta||\eta|)dx\leq\frac{1}{18} \|\Delta\eta\|_2^2 +C(\|\nabla\zeta\|_2^2+\|\eta\|_2^2).
$$
Substituting the previous three inequalities into (\ref{5.6}) yields
\begin{align}
  &\frac12\frac{d}{dt}(\|\nabla\eta\|_2^2+\|\eta\|_2^2)+\|\Delta\eta\|_2^2 +\|\nabla\eta\|_2^2\nonumber\\
  \leq&\frac16\|(\Delta\xi,\Delta\eta,\nabla\xi,\nabla\eta)\|_2^2 +C(1+\|(v_1, u_2)\|_4^4\nonumber\\
  &+\|(\nabla v_1,\nabla u_2)\|_2^4)\|(\nabla\eta,\nabla\zeta,\xi,\eta)\|_2^2. \label{5.7}
\end{align}

Applying the operator $(I-\Delta)^{-1}$ to equation (\ref{5.5}) yields
$$
\partial_t\zeta=-(I-\Delta)^{-1}\text{div}\,(u_1\theta+u\theta_2+v).
$$
Taking the $L^2(\mathbb R^2)$ inner product to the above equation with $(I-\Delta)\zeta$, it follows from integration by parts that
$$
\frac{1}{2}\frac{d}{dt}(\|\nabla\zeta\|_2^2+\|\zeta\|_2^2)=\int_{\mathbb R^2}(u_1\theta+u\theta_2+v)\cdot\nabla\zeta dx.
$$
Integration by parts yields
$$
\int_{\mathbb R^2}u_1\theta\cdot\nabla\zeta dx=\int_{\mathbb R^2}u_1(\zeta-\Delta\zeta)\cdot\nabla\zeta dx=\int_{\mathbb R^2}\partial_iu_1\partial_i\zeta\cdot\nabla\zeta dx\leq\|\nabla u_1\|_\infty\|\nabla\zeta\|_2^2.
$$
By the H\"older and Cauchy inequalities, we have
\begin{align*}
\int_{\mathbb R^2}(u\theta_2+v)\cdot\nabla\zeta dx\leq&\int_{\mathbb R^2}[(|\Delta\xi|+|\xi|)|\theta_2|+|\Delta\eta|+|\eta|]|\nabla\zeta|dx\\
  \leq&\frac{1}{6}\|(\Delta\xi,\Delta\eta)\|_2^2+C(1+\|\theta_2\|_\infty^2) \|(\nabla\zeta,\xi,\eta)\|_2^2.
\end{align*}
Therefore, one obtains
\begin{equation}
  \label{5.8}
  \frac{1}{2}\frac{d}{dt}\|(\nabla\zeta,\zeta)\|_2^2\leq \frac{1}{6}\|(\Delta\xi,\Delta\eta)\|_2^2 +C(1+\|\theta_2\|_\infty^2+\|\nabla u_1\|_\infty) \|(\nabla\zeta,\xi,\eta)\|_2^2.
\end{equation}

Summing inequalities (\ref{5.5}), (\ref{5.7}) and (\ref{5.8}) up, and using the Ladyzhenskaya inequality, one has
\begin{align*}
  &\frac{d}{dt}\|(\xi,\eta,\zeta)\|_{H^1}^2+\|(\nabla\xi,\nabla\eta, \nabla\zeta)\|_{H^1}^2\\
  \leq&(1+\|\theta_2\|_\infty^2+\|\nabla u_1\|_\infty+\|(u_1, u_2, v_1, v_2)\|_4^4+\|(\nabla u_2, \nabla v_1)\|_2^4)\|(\xi,\eta,\zeta)\|_{H^1}^2\\
  \leq&(1+\|\theta_2\|_\infty^2+\|\nabla u_1\|_\infty+\|(u_1, u_2, v_1, v_2)\|_2^2\|(\nabla u_1,\nabla u_2,\nabla v_1,\nabla v_2)\|_2^2\\
  &+\|(\nabla u_2, \nabla v_1)\|_2^4)\|(\xi,\eta,\zeta)\|_{H^1}^2,
\end{align*}
from which, recalling the regularities of the strong solutions in Theorem \ref{theorem}, and applying the Gronwall inequality, the conclusion follows.
\end{proof}

\section*{Acknowledgments}
The authors would like to dedicate this work to Professor Peter Lax on the occasion of his 90th birthday as a token of great respect and admiration. The work is supported in part by a grant of the ONR, and by the  NSF
grants DMS-1109640 and DMS-1109645.

\end{document}